\documentclass[preprint,10pt]{elsarticle} 
\makeatletter
    \def\ps@pprintTitle{%
  \def\@oddfoot{\reset@font\hfil\thepage\hfil\textit{\today}}
                       }
\makeatother
\usepackage[T1,IL2]{fontenc}
\usepackage[english]{babel}
\usepackage[cp1250]{inputenc}
\usepackage{amsthm}%\begin{proof}...\end{proof}
\usepackage{amsmath, amssymb}
\usepackage{enumerate}
\usepackage{hyperref}
\usepackage{color}
\newtheorem{theorem}{Theorem}
\newtheorem{lemma}[theorem]{Lemma} 
\newtheorem{corollary}[theorem]{Corollary}
\newtheorem{example}[theorem]{Example}
\newdefinition{definition}[theorem]{Definition}

\newdefinition{note}[theorem]{Note} 
\def\emphasize#1{\textit{#1}}
\newcommand{\dhalf}{\frac{d}{2}}                              %d/2
\newcommand{\dmhalf}{\frac{d-1}{2}}                           %(d-1)/2
\newcommand{\floor}[1]{\left\lfloor #1\right\rfloor}          %dolna cela cast
\newcommand{\lrexp}[2]{\left( #1\right)^{#2}}                 %okruhle zatvorky okolo s exponentom
\newcommand{\ke}{{q,\epsilon}}                                %dolny index q,epsilon
\newcommand{\mb}[1]{M\!B_{d,#1}}                              %bound for bipartite Moore graphs
\newcommand{\wtimes}{-\mathrm{times}}
\def\correctedTV{black}
\def\correctedMA{black}
\def\correctedMAa{black}
\def\correctedMAu{black}
\journal{Discrete Mathematics}
\begin{document} 
\hyphenation{ge-ne-ra-tors ge-ne-ra-ting ge-ne-ra-ted}
\setcounter{MaxMatrixCols}{20}
%%%%%%%%%%%%%%%%%%%%%%%%%%%%%%%%%%%%%%%%%%%%%%%%%%%%%%%%%%%%%%%%%%%%%%%%%%%%%%%%%%%%%%%%
%%%%%%%%%%%%%%%%%%%%%%%%%%%%%%%%%%%%%%%%%%%%%%%%%%%%%%%%%%%%%%%%%%%%%%%%%%%%%%%%%%%%%%%%
\begin{frontmatter}
\iffalse
\title{% 
\vglue-100pt\hfill 
{\fontsize{10}{12}\selectfont Submitted to \textit{Discrete Mathematics} January 22, 2015; revised February 1, 2016\vglue70pt}% 
Large Cayley digraphs and bipartite Cayley digraphs of odd diameters}%
%Submitted to Found. Phys. April 7 2007, revised September 6
%January 22, 2015, submitted to Discrete Mathematics
\fi
\title{Large Cayley digraphs and bipartite Cayley digraphs of odd diameters}

\author[ma]{Marcel~Abas\corref{cor1}}
\ead{abas@stuba.sk}

\author[tv]{Tom\'a\v s~Vetr\'\i k}
\ead{vetrikt@ufs.ac.za}

\cortext[cor1]{Corresponding author}
\address[ma]
        {Institute of Applied Informatics, Automation and Mechatronics,\\
         Faculty of Materials Science and Technology in Trnava,\\
         Slovak University of Technology in Bratislava,
         Trnava, Slovakia}
\address[tv]
        {Department of Mathematics and Applied Mathematics,\\
         University of the Free State, Bloemfontein, South Africa}              

\begin{abstract}\label{abstract}\noindent 
Let $C_{d,k}$ be the largest number of vertices in a Cayley digraph of degree $d$ and diameter $k$, and let $BC_{d,k}$ be the largest order of a bipartite Cayley digraph 
for given $d$ and $k$.
For every degree $d\geq2$ and for every odd $k$ we construct Cayley digraphs of order $2k\left(\lfloor\frac{d}{2}\rfloor\right)^k$ and diameter at most $k$, where $k\ge 3$, and bipartite Cayley digraphs of order $2(k-1)\left(\lfloor\frac{d}{2}\rfloor\right)^{k-1}$ and diameter at most $k$, where $k\ge 5$.
These constructions yield the bounds $C_{d,k} \ge 2k\left(\lfloor\frac{d}{2}\rfloor\right)^k$ for odd $k\ge 3$ and $d\ge \frac{3^{k}}{2k}+1$,
and $BC_{d,k} \ge 2(k-1)\left(\lfloor\frac{d}{2}\rfloor\right)^{k-1}$ for odd $k\ge 5$ and $d\ge \frac{3^{k-1}}{k-1}+1$.
Our constructions give the best currently known bounds on the orders of large Cayley digraphs and bipartite Cayley digraphs of given degree and odd diameter $k\ge 5$.
In our proofs we use new techniques based on properties of group automorphisms of direct products of abelian groups.
\end{abstract}


\end{frontmatter}
%%%%%%%%%%%%%%%%%%%%%%%%%%%%%%%%%%%%%%%%%%%%%%%%%%%%%%%%%%%%%%%%%%%%%%%%%%%%%%%%%%%%%%%%%%%%%%%%%%%%%%%%%%%%
%%%%%%%%%%%%%%%%%%%%%%%%%%%%%%%%%%%%%%%%%%%%%%%%%%%%%%%%%%%%%%%%%%%%%%%%%%%%%%%%%%%%%%%%%%%%%%%%%%%%%%%%%%%%
\section{Introduction}

The study of large graphs of given degree and diameter has often been restricted to special classes of graphs. 
{\color{\correctedTV} A particularly useful class is that of Cayley graphs, since their inherent symmetry enables to calculate their diameter by determining distances just from one vertex.}

A directed Cayley graph (or simply Cayley digraph) $G=Cay(\Gamma,X)$ is specified by an \emphasize{underlying} group $\Gamma$ and by a unit-free \emphasize{generating set} $X$ for this group.
Vertices of $Cay(\Gamma,X)$ are the elements of $\Gamma$, that is the vertex set $V(G)$
{\color{\correctedTV} is equal to}
$\Gamma$, and there is a directed edge from the vertex $u$ to the vertex $v$ if and only if there is a generator $x \in X$ such that $ux = v$. Note that in a Cayley digraph all vertices have the same in-degree and out-degree, thus we will call this number the degree of the Cayley digraph. Since the mapping 
$\varphi_h: g\to hg, g\in V(G)$ is a digraph automorphism for every $h\in\Gamma$, Cayley digraphs are vertex-transitive.

We study Cayley digraphs of large order for given degree and diameter. Let $C_{d,k}$ be the largest number of vertices in a Cayley digraph of degree
$d$ and diameter $k$. Clearly, the number of vertices in a digraph of maximum degree $d$ and diameter $k$ can not exceed the 
\emphasize{Moore bound} $M_{d,k} = 1 + d + d^2 + \dots + d^k$ {\color{\correctedTV} and} therefore $C_{d,k} \le M_{d,k}$.
Complete digraphs are the largest Cayley digraphs of diameter $1$. They yield the result $C_{d,1} = M_{d,1}=d+1$.
\v Zd\'\i malov\'a and Stanekov\'a \cite{ZS2010} studied vertex-transitive digraphs of \cite{FMC1993} and showed that $C_{d,2} = d^2 + d$ for $d = q-1$, where $q \ge 3$ is a prime power, and $C_{d,3} \ge d^3 - d$, where $d \ge 3$ is a prime power. 
{\color{\correctedTV} More general mixed Cayley graphs of diameter $2$ were studied by \v Siagiov\'a \cite{s}.} 
In \cite{V2012} it was proved that $C_{d,k}\ge k \left(\floor{\frac{d}{2}}\right)^k$ for any $k\ge 3$ and $d\ge 4$. We improve this result for odd diameters. For odd $k\geq3$ and for $d\geq2$ we construct Cayley digraphs of order $2k\left(\floor{\frac{d}{2}}\right)^k$, degree $d$ and diameter at most $k$, and we show that $C_{d,k} \ge 2k\left(\floor{\frac{d}{2}}\right)^k$ for each $d\geq\frac{3^{k}}{2k}+1$.

Now let us consider bipartite Cayley digraphs. Let $BC_{d,k}$ denote the largest possible number of vertices in a bipartite Cayley
digraph of degree $d$ and diameter $k$. Aider \cite{A1992} showed that the number of vertices in a bipartite digraph is at most
$2(1+d^2+ \dots +d^{k-1})$ if $k$ is odd, and $2d(1+d^2+ \dots +d^{k-2})$ if $k$ is even. The largest Cayley digraphs of diameter $2$ are the complete bipartite digraphs, thus $BC_{d,2}=2d$ for any $d \ge 2$.  Constructions of bipartite Cayley digraphs presented in \cite{V2012} yield the bounds
$BC_{d,3}\ge 2d^2$ for $d\ge 2$, $BC_{d,k}\ge 2(k-1)\left(\floor{\frac{d}{2}}\right)^{k-1}$ if $k \ge 4$ is even and $d \ge 4$ and
$BC_{d,k}\ge (k-1)\left(\floor{\frac{d}{2}}\right)^{k-1}$ if $k\ge 5$ is odd and $d \ge 4$. 
We present a construction of bipartite Cayley digraphs of order $2(k-1)\left(\floor{\frac{d}{2}}\right)^{k-1}$, degree $d\geq2$ and diameter at most $k$, where $k\geq5$ is odd. For $d>\frac{3^{k-1}}{k-1}+1$ we obtain the bound $BC_{d,k} \ge 2(k-1)\left(\floor{\frac{d}{2}}\right)^{k-1}$.
%%%%%%%%%%%%%%%%%%%%%%%%%%%%%%%%%%%%%%%%%%%%%%%%%%%%%%%%%%%%%%%%%%%%%%%%%%%%%%%%%%%%%%%%%%%%%%%%%%%%%%%%%%%%
%%%%%%%%%%%%%%%%%%%%%%%%%%%%%%%%%%%%%%%%%%%%%%%%%%%%%%%%%%%%%%%%%%%%%%%%%%%%%%%%%%%%%%%%%%%%%%%%%%%%%%%%%%%%
\section{Preliminaries}\label{sec_preliminaries}

Let $H$ be any additive abelian group of order $n$ with unit element $0$ and let $H^k=H\times H\times\cdots\times H$ be the direct product of $k$ copies of $H$. Elements of $H^k$ {\color{\correctedTV} will be written in the form}
$\vec{h}=(h_1,h_2,\dots,h_k),\ h_i\in H$, $i\in\{1,2\dots,k\}$
{\color{\correctedTV} and the group operation will be defined by}
$\vec{h}\cdot\vec{h}'=(h_1,h_2,\dots,h_k)\cdot(h_1',h_2',\dots,h_k')=(h_1+h_1',h_2+h_2',\dots,h_k+h_k')$.

\begin{lemma}\label{lem_T_Dk}
Let $k=2q+\epsilon$, where $q\geq1$ is an integer, $\epsilon\in\{0,1\}$, $k\ne2$, and let $Sym(k)$ be the symmetric group on $k$ symbols $1,2,\dots,k$. 
Let $A_\ke=(1,2,3,\dots,k)$ be a $k$-cycle and let $B_\ke$ be products of the following disjunct transpositions:\\
%$B_{q,0}=(1,q)(2,q-1)\dots(\frac{q}{2},\frac{q}{2}+1)(q+1,k)(q+2,k-1)\dots(\frac{k+q}{2},\frac{k+q}{2}+1)$,\\
$B_{q,0}=(1,q)(2,q-1)\dots(\frac{q}{2},\frac{q}{2}+1)(q+1,k)(q+2,k-1)\dots(\frac{3q}{2},\frac{3q}{2}+1)$,\\
%$B_{q,1}=(1,q)(2,q-1)\dots(\frac{q}{2},\frac{q}{2}+1)(q+1,k)(q+2,k-1)\dots(\frac{k+q+1}{2}-1,\frac{k+q+1}{2}+1)$,\\
$B_{q,1}=(1,q)(2,q-1)\dots(\frac{q}{2},\frac{q}{2}+1)(q+1,k)(q+2,k-1)\dots(\frac{3q}{2},\frac{3q}{2}+2)$,\\
for $q$ even, and\\
%$B_{q,0}=(1,q)(2,q-1)\dots(\frac{q+1}{2}-1,\frac{q+1}{2}+1)(q+1,k)(q+2,k-1)\dots(\frac{k+q+1}{2}-1,\frac{k+q+1}{2}+1)$,\\
$B_{q,0}=(1,q)(2,q-1)\dots(\frac{q+1}{2}-1,\frac{q+1}{2}+1)(q+1,k)(q+2,k-1)\dots(\frac{3q+1}{2}-1,\frac{3q+1}{2}+1)$,\\
%$B_{q,1}=(1,q)(2,q-1)\dots(\frac{q+1}{2}-1,\frac{q+1}{2}+1)(q+1,k)(q+2,k-1)\dots(\frac{k+q}{2},\frac{k+q}{2}+1)$,\\
$B_{q,1}=(1,q)(2,q-1)\dots(\frac{q+1}{2}-1,\frac{q+1}{2}+1)(q+1,k)(q+2,k-1)\dots(\frac{3q+1}{2},\frac{3q+1}{2}+1)$,\\
for $q$ odd.\\ 
Then the group $T_\ke=\langle A_\ke, B_\ke\rangle$ is isomorphic to the dihedral group $D_k$ of order $2k$.
\end{lemma}

\begin{proof}
The dihedral group $D_k$ of order $2k$ has the standard presentation in the form $D_k=\langle A,B\vert A^k=B^2=1, BAB=A^{-1}\rangle$. It is easy to verify that $A^k_\ke=I$, where $k$ is the true order of $A_\ke$, $B^2_\ke=I$ and that 
$B_\ke A_\ke B_\ke=(A_\ke)^{-1}$. Since $A^i_\ke\ne A^j_\ke$ and 
$A^i_\ke B_\ke\ne A^j_\ke B_\ke$ for $i,j\in\{0,1,\dots,k-1\}$, $i\ne j$, the group $T_\ke$ has order $2k$. Therefore 
{\color{\correctedTV} the mapping} $\psi: T_\ke\to D_k$ given by $A_\ke\to A$ and $B_\ke\to B$ is a group isomorphism.
\end{proof}

\begin{example}\label{ex_epsilon}
For the first four values of $k$ we have:\\
$k=3$ ($q=1,\epsilon=1$), $A_{1,1}=(123)$, $B_{1,1}=(23)$,\\
$k=4$ ($q=2,\epsilon=0$), $A_{2,0}=(1234)$, $B_{2,0}=(12)(34)$,\\
$k=5$ ($q=2,\epsilon=1$), $A_{2,1}=(12345)$, $B_{2,1}=(12)(35)$,\\
$k=6$ ($q=3,\epsilon=0$), $A_{3,0}=(123456)$, $B_{3,0}=(13)(46)$.
\end{example}

In what {\color{\correctedTV} follows we} will write $D_k$ instead of $T_\ke$, where $D_k=T_{q,0}$ for $k$ even and $D_k=T_{q,1}$ for odd $k$. 
Let $\Gamma_k=H^k\rtimes_\varphi D_k$ be a semidirect product of the group $H^k$ and the group $D_k$  
{\color{\correctedTV} represented as} 
$T_\ke$. 
Note that to simplify the notation we will often omit the subscript $k$.
Elements of $\Gamma$ 
{\color{\correctedTV} will be written in the form} $\vec{h}\cdot C$, $\vec{h}\in H^k$, $C\in D_k$ and the product of two elements of $\Gamma$ is given by:
\begin{align}\label{eq_product}
\nonumber
\vec{h}C\cdot\vec{h}'C' &= (h_1,h_2,\dots,h_k)C\cdot(h_1',h_2',\dots,h_k')C'\\
                        &=(h_1+h'_{C(1)},h_2+h'_{C(2)},\dots,h_k+h'_{C(k)})CC',
\end{align}
\noindent
where $CC'$ is the product of $C$ and $C'$ in $D_k$.
{\color{\correctedTV} In the homomorphism $\varphi: D_k\to Aut(H^k)$, elements of $D_k$ induce permutations of {\color{\correctedMA} coordinates of} $\vec{h}$.}
If the order of $H$ is $n$, then the group $\Gamma$ has the order $\vert\Gamma\vert=2kn^k$. For an element $g=\vec{h}C$ of $\Gamma$ we say that $\vec{h}$ is the 
\emphasize{prefix} $P(g)$ of $g$ and that $C$ is the \emphasize{suffix} $S(g)$ of $g$. 

{\color{\correctedTV} In what follows} we will use the following notation for some special elements of $\Gamma$:
\begin{align}\label{eq_a_x}
%\nonumber
%\tilde{a}(x) &= (x,0,\dots,0),\ x\in H,\ \tilde{a}(x)\in H^k\\
        a(x) &= (x,0,\dots,0)A,\ x\in H,
\end{align}
where $A=A_\ke$ is the element of $T_\ke$ defined in Lemma \ref{lem_T_Dk}, and
\begin{align}\label{eq_b_x}
%\nonumber
%\tilde{b}(x) &= (x,0,\dots,0,x,0,\dots,0),\ x\in H,\ \tilde{b}(x)\in H^k\\
        b(x) &= (x,0,\dots,0,x,0,\dots,0)B,\ x\in H,
\end{align}
{\color{\correctedTV} where $x$ occurs in the first and $(q+1)$-st coordinate}
and $B=B_\ke$ is the element of $T_\ke$ defined in Lemma \ref{lem_T_Dk}.

Let $\Gamma=H^k\rtimes_\varphi D_k$ be the underlying group and let $X=\{a(x),b(x)\vert x\in H\}$ be the generating set for the Cayley digraph
$G=Cay(\Gamma,X)$. 
{\color{\correctedTV} The graph $G$ has degree $d=\vert X\vert=2n$ and order $2kn^k = 2k\left(\frac{d}{2}\right)^k$, if $d$ is even.}

{\color{\correctedMAa}
Let $w=C_1C_2\ldots C_k$, $C_i\in\{A,B\}$, $i\in\{1,2,\dots,k\}$, be a word in $D_k$.
We say that the product $C=C_1\cdot C_2\cdot\ldots\cdot C_k=V(w)$ is the \emphasize{value} of the word $w$.}
\iffalse
{\color{\correctedMA}
Let $w=C_1C_2\ldots C_k$, $C_i\in\{A,B\}$, $i\in\{1,2,\dots,k\}$, be a word in $D_k$. 
The word $w$ represents an element $C$ of the group $D_k$, namely the product $C_1C_2\ldots C_k$. We say that $C=V(w)$ is the \emphasize{value} of the word $w$.} 
\fi
%The \emphasize{value} $V(w)$ of the word $w$ is the product $V(w)=C_1\cdot C_2\cdot\dots\cdot C_k$.
The \emphasize{fiber} over the word $w$ is the set of all sequences $\gamma_1,\gamma_2,\dots,\gamma_k$ such that $\gamma_i=a(x_i)$ if $C_i=A$ 
and $\gamma_i=b(x_i)$ if $C_i=B$, $x_i\in H$. We denote the set of sequences by $\vec{\gamma}_w$ or simply by $\vec{\gamma}$. For given 
$\vec{x}\in H^k$, the \emphasize{{\color{\correctedMA} value}} of $\vec{\gamma}_w(\vec{x})$ is the product 
$V(\vec{\gamma}(\vec{x}))=\gamma_1(x_1)\cdot\gamma_2(x_2)\dots\cdot\gamma_k(x_k)$.

Note that the previous definitions have a simple interpretation.
%in the language of voltage assignments (cf. \cite{G1974}). 
If $G=Cay(\Gamma,X)$ is a Cayley digraph and $w$ is a word in $D_k$, then the fiber $\vec{\gamma}_w$ over $w$ is the set of all oriented walks of length $k$ starting at $1_\Gamma$ and such that the vertices on the walks have suffices {\color{\correctedMAa} $C_1,C_1\cdot C_2,C_1\cdot C_2\cdot C_3,\dots,C=V(w)$}, respectively.

Since $V(\vec{\gamma}_w(\vec{x}))$ is an element of $\Gamma$, it has the form $V(\vec{\gamma}_w(\vec{x}))=\vec{z}\cdot C$, where $\vec{z}\in H^k$ and 
$C=V(w)\in D_k$. It is easy to see that for $\vec{z}$ we have $\vec{z}=(z_1,z_2,\dots,z_k)$, where 
$z_i=m_{i,1}x_1+m_{i,2}x_2+\dots+m_{i,k}x_k$, such that $m_{i,j}$ is either $0$ or $1$ and that $m_{1,j}+m_{2,j}+\dots+m_{k,j}$ is equal to $1$ if $\gamma_i$ is $a(x_i)$ or 2 if $\gamma_i$ is $b(x_i)$. Now, let
\begin{align}\label{eq_V}
\nonumber
V(\vec{\gamma}_w(\vec{x})) =(&m_{1,1}x_1+m_{1,2}x_2+\dots+m_{1,k}x_k,\\
\nonumber
                           &m_{2,1}x_1+m_{2,2}x_2+\dots+m_{2,k}x_k,\\
\nonumber                  &\vdots\\
                           &m_{k,1}x_1+m_{k,2}x_2+\dots+m_{k,k}x_k)\cdot C.                        
\end{align}
The \emphasize{fiber matrix over the word} $w$ is the square matrix $M=M(w)$ of degree $k$ and of the form
\begin{equation}\label{eq_matrix}
M(w)=
\begin{pmatrix}
  m_{1,1} & m_{1,2} & \hdots & m_{1,k} \\
  m_{2,1} & m_{2,2} & \hdots & m_{2,k} \\
  \vdots & \vdots & \ddots & \vdots \\
  m_{k,1} & m_{k,2} & \hdots & m_{k,k}
 \end{pmatrix}.
\end{equation}

Clearly, $i$-th row of $M$ corresponds to the $i$-th coordinate of $\vec{z}$ and $j$-th column of $M$ corresponds to the $j$-th generator of $\vec{\gamma}_w(\vec{x})$.

{\color{\correctedMA}
\begin{example}\label{ex_matrix}
Let $q=2$ and $\epsilon=1$. Then $k=2q+1=5$, $a(x)=(x,0,0,0,0)A=(x,0,0,0,0)(12345)$ and $b(x)=(x,0,x,0,0)B=(x,0,x,0,0)(12)(35)$. If, for example, $w=ABBAB$, then 
$\vec{\gamma}_w(\vec{x})=a(x_1)A,b(x_2)B,b(x_3)B,a(x_4)A,b(x_5)B$ and after a computation we have $V(\vec{\gamma}_w(\vec{x}))=a(x_1)A\cdot b(x_2)B \cdot b(x_3)B \cdot a(x_4)A\cdot b(x_5)B=(x_1+x_3,x_2+x_4,x_3+x_5,x_2,x_5)\cdot A^2B$. The fiber matrix over the word $w$ is
$M(w)=
\begin{pmatrix}
  1 & 0 & 1 & 0 & 0 \\
  0 & 1 & 0 & 1 & 0 \\
  0 & 0 & 1 & 0 & 1 \\
  0 & 1 & 0 & 0 & 0 \\
  0 & 0 & 0 & 0 & 1 
 \end{pmatrix}.
 $
\end{example}}

{\color{\correctedTV} One sees} that for given word $w$ and for the fiber matrix $M=M(w)$ over $w$ we have $P(V(\vec{\gamma}_w(\vec{x})))=(M\vec{x}^T)^T=\vec{x}M^T$ (note that $P(g)$ is the prefix of $g$).

\begin{definition}\label{def_cover}
We say that the fiber $\vec{\gamma}_w$ over the word $w$ \emphasize{covers} the element $C=V(w)$ of $D_k$ if for every
$\vec{y}=(y_1,y_2,\dots,y_k)\in H^k$ there is 
a vector $\vec{x}=(x_1,x_2,\dots,x_k)\in H^k$ such that $M\vec{x}^T=\vec{y}^T$.
\end{definition}

The definition says that $\vec{\gamma}_w$ covers $C=V(w)$ if for every vertex $g$ of $G=Cay(\Gamma,X)$ with suffix $S(g)=C$ there is an oriented walk of length $k$ from $1_\Gamma$ to the vertex $g$ such that the vertices on the walk have suffices {\color{\correctedMAa} $C_1,C_1\cdot C_2,\dots,C=V(w)$}.

\begin{lemma}\label{lem_one_x_cover}
Let $\vec{\gamma}_w$ cover $C=V(w)$ and let $\vec{y}_0\in H^k$. Then there exists exactly one $\vec{x}_0\in H^k$ such that 
$M\vec{x}_0^T=\vec{y}_0^T$.
\end{lemma}

\begin{proof}
Since $\vec{\gamma}_w$ covers $C$, for every $\vec{y}_0\in H^k$ there exists a vector $\vec{x}_0\in H^k$ such that 
$M\vec{x}_0^T=\vec{y}_0^T$. Now we show that this $\vec{x}_0$ is unique.

In the Cayley digraph $G=Cay(\Gamma,X)$ there are exactly $n^k$ vertices of the form $\vec{x}C$. The lemma in fact says that for every vertex $g\in G$ of the form $g=\vec{x}C$ there is exactly one oriented walk from $1_\Gamma$ to $g$ such that the suffices of the 
vertices $g_1,g_2,\dots,g_k$, where $g_k=g$, on the walk are the elements {\color{\correctedMAa} $C_1,C_1\cdot C_2,\dots,C=V(w)$} of $D_k$. Since the group $H$ has order $n$, there is at most $n^k$ terminal vertices of $\vec{\gamma}_w$. On the other hand, every vertex with suffix $C$ (its number is also $n^k$) is 
reachable from $1_\Gamma$ by a walk $\vec{\gamma}_w(\vec{x}_0)$ (since $\vec{\gamma}$ covers $C$).
\end{proof}

\begin{corollary}\label{cor_one_y_cover}
Let $\vec{\gamma}_w$ {\color{\correctedTV} cover} $C=V(w)$ and let $M$ be the fiber matrix over $w$. Then the mapping $\mu: H^k\to H^k$ given by $\mu(\vec{x})=\vec{x}M^T$ is a group automorphism.
\end{corollary}

\begin{proof}
By Definition \ref{def_cover}, for every $\vec{y}\in H^k$ there is a vector $\vec{x}\in H^k$ such that $\mu(\vec{x})=\vec{y}$ and by Lemma \ref{lem_one_x_cover} the mapping $\mu$ is a bijection. It is easy to see that for $\vec{x}=\vec{x}_1+\vec{x}_2$ we have 
$\mu(\vec{x}_1+\vec{x}_2)=\mu(\vec{x})=\vec{x}M^T=(\vec{x}_1+\vec{x}_2)M^T=\vec{x}_1M^T+\vec{x}_2M^T=\mu(\vec{x}_1)+\mu(\vec{x}_2)$. Therefore $\mu$ is a bijective group endomorphism of the group $H^k$ and therefore a group automorphism.
\end{proof}

\begin{definition}\label{def_type}
We say that a column $j$ (row $i$) of a matrix $M$ is of \emphasize{type I} if it contains exactly one unit and it is of \emphasize{type II} 
if it contains two (two or more) $1$'s.
\end{definition}

\begin{lemma}\label{lem_type_I}
Let $M$ be {\color{\correctedTV} a zero-one matrix} of order $k\geq 1$ with $det(M)=\pm1$ such that in any column of $M$ there are at most two $1$'s. 
Then there is at least one row of type I in $M$.
\end{lemma}

\begin{proof} 
{\color{\correctedTV} Let $l\in \{ 0, \dots , k \}$ be} the number of columns of type II in $M$. The assertion is true for $k=1$ and for every $k>1$, $l<k$. For $k=l=2$ the determinant {\color{\correctedTV} of $M$} is equal to $0$. Now let $l=k\geq3$. If there is no row of type I, then in every row of the matrix there are exactly two '1'-s. Therefore we have a zero-one matrix of order $k\geq 3$ with equal row and column sum $2$. It is easy to verify that the determinant of such a matrix is an even number -- a contradiction.
\end{proof}
 
\begin{lemma}\label{lem_regular_covers}
Let $w$ be a word in $D_k$ and let the fiber matrix $M=M(w)$ be a matrix with $det(M)=\pm1$. Then $\vec{\gamma}_w$ covers $C=V(w)$. 
\end{lemma}

\begin{proof}
It is sufficient to show that for every pair of different elements $\vec{x}_0\ne\vec{x}_1$ of $H^k$ we have $M\vec{x}_0^T\ne M\vec{x}_1^T$. We prove the lemma by a contradiction. Let $\vec{x}_0\ne\vec{x}_1$ and $\vec{y}_0$ in $H^k$ be such that $M\vec{x}_0^T=M\vec{x}_1^T=\vec{y}_0^T$. That is, $M\vec{x}^T=\vec{0}^T$ for some non-zero vector $\vec{x}$ of $H^k$. Since $M$ is a zero-one matrix with determinant $\pm1$ and such that there are at most two {\color{\correctedTV} $1$'s} in any column of $M$, there is (by Lemma \ref{lem_type_I}) at least one row, say $i$, of type I in $M$. Let the {\color{\correctedTV} 1} of this row be in a column $j$. 
{\color{\correctedTV} The $1$ is the only $1$ in the row, thus} from the definition of fiber matrix it follows that to get $y_i=0$ we have to set $x_j=0$.  Now we can omit the row $i$ and the column $j$ to get a new matrix $M'$ (corresponding to a new mapping $M': H^{k-1}\to H^{k-1}$). Since $M'$ is again a zero-one matrix with determinant $\pm1$ and with at most two $1$'s in any column, there is at least one row of type I in $M'$ and we can continue with the process. In this way, after $(k-1)$ steps, we get $1\times 1$ matrix $\tilde{M}$ with entry $\tilde{m}_{1,1}=1$. The entry corresponds to a row $i'$ and column $j'$ of $M$. To obtain $y_{i'}=0$ we have to set $x_{j'}=0$. Therefore $\vec{x}$ is the zero vector of $H^k$ -- a contradiction.
\end{proof}

\begin{note}
We have already seen that if $\vec{\gamma}_w$ covers $V(w)$, then the mapping $\mu(\vec{x})=\vec{x}M^T$ is a group automorphism. Lemma 
\ref{lem_regular_covers} in fact says, that if $M$ {\color{\correctedTV} has $det(M)=\pm1$, then the mapping $\mu(\vec{x})=\vec{x}M^T$ is a group automorphism 
{\color{\correctedMA} as well}.}
\end{note}
%%%%%%%%%%%%%%%%%%%%%%%%%%%%%%%%%%%%%%%%%%%%%%%%%%%%%%%%%%%%%%%%%%%%%%%%%%%%%%%%%%%%%%%%%%%%%%%%%%%%%%%%%%%%
%%%%%%%%%%%%%%%%%%%%%%%%%%%%%%%%%%%%%%%%%%%%%%%%%%%%%%%%%%%%%%%%%%%%%%%%%%%%%%%%%%%%%%%%%%%%%%%%%%%%%%%%%%%%
\section{Results}\label{sec_results}

In this section we study products of $k=2q+\epsilon,\ \epsilon\in\{0,1\}$, generators (elements of $X$) 
in terms of matrices $M(w)$ over the corresponding word $w$. {\color{\correctedTV} More precisely,} for a given set $S$ of elements of $D_k$, for each $C\in S$ we find a word $w$ in $D_k$ such that $V(w)=C$ and the fiber matrix $M$ over $w$ has determinant equal to $\pm1$.

To simplify the notation, {\color{\correctedTV} in the lemmas that follow we will write $a(i)$ for the generator which} yields the column of $M$ with the $i$-th entry equal to 
$1$, $1 \le i \le k$. Similarly, we will write $b(j,j')$ for the generator with the $j$-th and $j'$-th entry equal to $1$, $1 \le j < j' \le k$.
For example, let us have $q=2$, $\epsilon=1$ and $w=ABBAB$ as in Example \ref{ex_matrix}. Then 
$a(x_1)b(x_2)b(x_3)a(x_4)b(x_5)$ corresponds to $(x_1+x_3,x_2+x_4,x_3+x_5,x_2,x_5)A^2B$ and we simply write 
$\vec{\gamma}_w\to a(1)b(2,4)b(1,3)a(2)b(3,5)$ instead. Note that $i,j$ and $j'$ in $a(i)$ and $b(j, j')$ are determined by the product of {\color{\correctedTV} preceding} generators. 

%%%%%%%%%%%%%%%%%%%%%%%%%%%%%%%%%%%%%%%%%%%%%%%%%%%%%%%%%%%%%%%%%%%%%%%%%%%%%%%%%%%%%%%%%%%%%%%%%%%%%%%%%%%%
%%%%%%%%%%%%%%%%%%%%%%%%%%%%%%%%%%%%%%%%%%%%%%%%%%%%%%%%%%%%%%%%%%%%%%%%%%%%%%%%%%%%%%%%%%%%%%%%%%%%%%%%%%%%
\subsection{Cayley digraphs}\label{subsec_digraphs}
In the next lemma, for any element $A^p B^l$ of the group $D_k$, $k$ odd, $p\in\{1,2,\dots,k\}$ and $l\in\{0,1\}$, we present a sequence of $k$ generators (a word in $D_k$) whose product is the element with suffix $A^p B^l$ and the corresponding matrix $M$ has determinant $\pm1$.

\begin{lemma}\label{lem_eps_1}
Let $q\geq1$ be an integer and let $k=2q+1$. Then for every $C$ in $D_k$ there is a word $w$ in $D_k$ such that $V(w)=C$ and the fiber matrix $M(w)$ over the word $w$ is {\color{\correctedTV} a matrix with} $det(M)=\pm1$.
\end{lemma}

\begin{proof}
All the elements of $D_k$ are of the form $A^pB^l$, $p\in\{1,2,\dots,k\}$, $l\in\{0,1\}$ and we distinguish eight cases i), ii),\dots, viii), depending on the values of $p$ and $l$. In the first four cases $V(w)=A^p$ and in the last four cases $V(w)=A^pB$.\\\\
%%%%%%%%%%%%%%%%%%%%%%%%%%%%%%%%%%%%%%%%%%%%%%%   1 i)   %%%%%%%%%%%%%%%%%%%%%%%%%%%%%%%%%%%%%%%%%%%%%%%%%%%
i) $l=0$, $p$ even, $q+1\leq p\leq 2q$, $V(w)=A^p$, where {\color{\correctedMAu} \\\phantom{i)} 
   $w=B\underbrace{AABB}_{[(q-p/2)\wtimes]}\underbrace{A}_{[(p-q)\wtimes]}B\underbrace{A}_{[(p-q-1)\wtimes]}$.}
\begin{eqnarray}\label{1_i}
\vec{\gamma}_w&\to& b(1, q+1) \nonumber \\
              && [a(q) a(q - 1) b(q-2, 2q-1)b(q-1, 2q) \nonumber \\            
              &&  a(q-2) a(q-3) b(q-4, 2q-3) b(q-3, 2q-2) \nonumber \\
              &&  \vdots \nonumber \\
              &&  a(p-q+2) a(p-q+1) b(p-q, p+1) b(p-q+1, p+2) ] \nonumber \\
              &&  a(p-q) a(p-q-1) \dots a(1) \nonumber \\
              &&  b(q+1, 2q+1) \nonumber \\
              &&  a(q+2) a(q+3) \dots a(p)
\end{eqnarray}
%%%%%%%%%%%%%%%%%%%%%%
Consider {\color{\correctedTV} a} column of $M$ {\color{\correctedTV} containing} only one entry $1$. If we remove the row and column in which this entry appears, we obtain a new square matrix of degree $2q = k-1$ whose determinant is equal to $\pm det(M)$. Since there are $(p-1)$ generators $a(i)$ in the expression
(\ref{1_i}), we have $(p-1)$ columns with the only non-zero entry. All these $(p-1)$ non-zero entries are in different rows. We delete rows and columns which correspond to these non-zero entries to get a new square matrix, say $M'$, of degree $2q-p+2$. We have $det(M') = \pm det(M)$. 
Note that the new matrix $M'$ contains the rows $q+1, p+1, p+2, \dots 2q+1$ of the matrix $M$ and the only generator which yields a column of $M'$ with two entries equal to $1$ is $b(q+1, 2q+1)$. The generator $b(q+1, 2q+1)$ corresponds to the last $(2q-p+2)$-th column of $M'$ and the $(q+1)$-th row of $M$ corresponds to the first row of $M'$. Except for the first row and the last column of $M'$, every row and every column of $M'$ contains exactly one non-zero entry, which implies that $det(M')= \pm 1$.
%%%%%%%%%%%%%%%%%%%%%%

{\color{\correctedTV} To give an example of this situation, let} $q=5$ and $p=6$. Then $k=11$,\\ $w=BAABBAABBAB$, $V(w)=A^6$,\\ 
$\vec{\gamma}_w\to b(1,6)a(5)a(4)b(3,9)b(4,10)a(3)a(2)b(1,7)b(2,8)a(1)b(6,11)$,\\\noindent

$M=\begin{pmatrix}
  1 & 0 & 0 & 0 & 0 & 0 & 0 & 1 & 0 & 1 & 0 \\
  0 & 0 & 0 & 0 & 0 & 0 & 1 & 0 & 1 & 0 & 0 \\
  0 & 0 & 0 & 1 & 0 & 1 & 0 & 0 & 0 & 0 & 0 \\
  0 & 0 & 1 & 0 & 1 & 0 & 0 & 0 & 0 & 0 & 0 \\
  0 & 1 & 0 & 0 & 0 & 0 & 0 & 0 & 0 & 0 & 0 \\
  1 & 0 & 0 & 0 & 0 & 0 & 0 & 0 & 0 & 0 & 1 \\
  0 & 0 & 0 & 0 & 0 & 0 & 0 & 1 & 0 & 0 & 0 \\
  0 & 0 & 0 & 0 & 0 & 0 & 0 & 0 & 1 & 0 & 0 \\
  0 & 0 & 0 & 1 & 0 & 0 & 0 & 0 & 0 & 0 & 0 \\
  0 & 0 & 0 & 0 & 1 & 0 & 0 & 0 & 0 & 0 & 0 \\
  0 & 0 & 0 & 0 & 0 & 0 & 0 & 0 & 0 & 0 & 1
\end{pmatrix}$
and
$M'=\begin{pmatrix}
  1 & 0 & 0 & 0 & 0 & 1 \\
  0 & 0 & 0 & 1 & 0 & 0 \\
  0 & 0 & 0 & 0 & 1 & 0 \\
  0 & 1 & 0 & 0 & 0 & 0 \\
  0 & 0 & 1 & 0 & 0 & 0 \\
  0 & 0 & 0 & 0 & 0 & 1 
\end{pmatrix}$.\\
%%%%%%%%%%%%%%%%%%%%%%%%%%%%%%%%%%%%%%%%%%%%%%%   1 ii)   %%%%%%%%%%%%%%%%%%%%%%%%%%%%%%%%%%%%%%%%%%%%%%%%%%%
\\
ii) $l=0$, $p$ even, $2\leq p\leq q$, $V(w)=A^p$, where {\color{\correctedMAu} \\\phantom{ii)} 
   $w=B\underbrace{A}_{[(q-p+2)\wtimes]}\underbrace{ABBA}_{[(p/2-1)\wtimes]}\underbrace{A}_{[(q-p+1)\wtimes]}B$.}
\begin{eqnarray}\label{1_ii}
\vec{\gamma}_w&\to& b(1, q+1) \nonumber \\
              &&  a(q) a(q-1) \dots a(p-1) \nonumber \\
              && [a(p-2) b(p-3, p+q-2) b(p-2, p+q-1) a(p-3) \nonumber \\            
              &&  a(p-4) b(p-5, p+q-4) b(p-4, p+q-3) a(p-5) \nonumber \\
              &&  \vdots \nonumber \\
              &&  a(2) b(1,q+2) b(2,q+3) a(1) ] \nonumber \\
              &&  a(2q+1) a(2q) \dots a(p+q+1) \nonumber \\
              &&  b(p,p+q)
\end{eqnarray}
%%%%%%%%%%%%%%%%%%%%%%
Let us note that non-zero entries in the matrix $M$ {\color{\correctedTV} corresponding} to the generators $a(i)$ appear in the rows: $1, 2, \dots , q$ and $p+q+1, p+q+2, \dots , 2q+1$. We remove these rows from $M$ together with columns {\color{\correctedTV} corresponding} to generators $a(i)$ to create a square matrix $M'$ of order $p$, such that
$det(M') = \pm det(M)$. 

For every generator $b(j,j')$ {\color{\correctedTV} appearing} in the expression
(\ref{1_ii}), $j$-th row does not appear in $M'$ and $j'$-th row appears in $M'$. It follows that each column of $M'$ contains exactly one non-zero entry. Since any two $j'$-s are different, each row of $M'$ contains exactly one non-zero entry  {\color{\correctedTV} as well.} Hence $det(M')=\pm1$.\\\\
%%%%%%%%%%%%%%%%%%%%%%%%%%%%%%%%%%%%%%%%%%%%%%%   1 iii)   %%%%%%%%%%%%%%%%%%%%%%%%%%%%%%%%%%%%%%%%%%%%%%%%%%%
iii) $l=0$, $p$ odd, $q+1\leq p\leq 2q+1$, $V(w)=A^p$, where {\color{\correctedMAu} \\\phantom{iii)} 
   $w=\underbrace{A}_{[(p-q-1)\wtimes]}\underbrace{AABB}_{[(q+1-(p+1)/2)\wtimes]}\underbrace{A}_{[(p-q)\wtimes]}B$.}
\begin{eqnarray}\label{1_iii}
\vec{\gamma}_w&\to& a(1) a(2) \dots a(p-q-1) \nonumber \\
              && [ a(p-q) a(p-q+1) b(p-q+2, p+2) b(p-q+1, p+1) \nonumber \\            
              &&  a(p-q+2) a(p-q+3) b(p-q+4, p+4) b(p-q+3, p+3) \nonumber \\
              &&  \vdots \nonumber \\
              &&  a(q-1) a(q) b(q+1, 2q+1)b(q, 2q) ] \nonumber \\              
              &&  a(q+1) a(q+2) \dots a(p)
\end{eqnarray}
%%%%%%%%%%%%%%%%%%%%%%
We study the matrix $M$ {\color{\correctedTV} corresponding} 
to this product of generators.
We focus on those entries $1$ in the matrix $M$, which lie in columns corresponding to the generators $a(i)$. We again create a new matrix $M'$ by removal of columns and rows of $M$, in which these entries $1$ appear to obtain a new square matrix of degree 
$2q-p+1$. Note that we removed first $p$ rows of $M$. It follows that
$det(M') = \pm det(M)$. 
The matrix $M'$ has exactly one non-zero entry in each column and in each row, 
{\color{\correctedTV} and so}
$det(M')= \pm 1$.\\\\
%%%%%%%%%%%%%%%%%%%%%%%%%%%%%%%%%%%%%%%%%%%%%%%   1 iv)   %%%%%%%%%%%%%%%%%%%%%%%%%%%%%%%%%%%%%%%%%%%%%%%%%%%
iv) $l=0$, $p$ odd, $1\leq p\leq q$, $V(w)=A^p$, where {\color{\correctedMAu}  \\\phantom{iv)} 
   $w=B\underbrace{A}_{[(q-p)\wtimes]}B\underbrace{A}_{[(q-p)\wtimes]}\underbrace{AABB}_{[((p-1)/2)\wtimes]}A$.}
\begin{eqnarray}\label{1_iv}
\vec{\gamma}_w&\to& b(1,q+1) \nonumber \\
              && a(q) a(q-1) \dots a(p+1) \nonumber \\
              && b(p, p+q+1) \nonumber \\
              && a(p+q+2) a(p+q+3) \dots a(2q+1) \nonumber \\
              && [ a(1) a(2) b(3,q+3) b(2,q+2) \nonumber \\            
              &&  a(3) a(4) b(5,q+5) b(4,q+4) \nonumber \\
              &&  \vdots \nonumber \\
              &&  a(p-2) a(p-1) b(p,p+q) b(p-1,p+q-1) ] \nonumber \\              
              &&  a(p)
\end{eqnarray}
%%%%%%%%%%%%%%%%%%%%%%%%%%%%%%%%%%%%%%%%%%%%%%%   1 v)   %%%%%%%%%%%%%%%%%%%%%%%%%%%%%%%%%%%%%%%%%%%%%%%%%%%
v) $l=1$, $p$ even, $q+1\leq p\leq 2q$, $V(w)=A^pB$, where {\color{\correctedMAu}  \\\phantom{v)} 
   $w=\underbrace{A}_{[(p-q)\wtimes]}\underbrace{AABB}_{[(q-p/2)\wtimes]}\underbrace{A}_{[(p-q)\wtimes]}B$.}
\begin{eqnarray}\label{1_v}
\vec{\gamma}_w&\to& a(1) a(2) \dots a(p-q) \nonumber \\
              && [ a(p-q+1) a(p-q+2) b(p-q+3, p+3) b(p-q+2, p+2) \nonumber \\            
              &&  a(p-q+3) a(p-q+4) b(p-q+5, p+5) b(p-q+4, p+4) \nonumber \\
              &&  \vdots \nonumber \\
              &&  a(q-1) a(q) b(q+1, 2q+1) b(q, 2q) ] \nonumber \\              
              &&  a(q+1) a(q+2) \dots a(p) \nonumber \\
              &&  b(p-q, q+1)
\end{eqnarray}
\clearpage
%%%%%%%%%%%%%%%%%%%%%%%%%%%%%%%%%%%%%%%%%%%%%%%   1 vi)   %%%%%%%%%%%%%%%%%%%%%%%%%%%%%%%%%%%%%%%%%%%%%%%%%%%
vi) $l=1$, $p$ even, $2\leq p\leq q$, $V(w)=A^pB$, where {\color{\correctedMAu}  \\\phantom{vi)} 
   $w=A\underbrace{AABB}_{[(p/2-1)\wtimes]}\underbrace{A}_{[(q-p+2)\wtimes]}B\underbrace{A}_{[(q-p+1)\wtimes]}$.}
\begin{eqnarray}\label{1_vi}
\vec{\gamma}_w&\to& a(1) \nonumber \\
              && [ a(2) a(3) b(4,q+4) b(3,q+3) \nonumber \\            
              &&  a(4) a(5) b(6,q+6) b(5,q+5) \nonumber \\
              &&  \vdots \nonumber \\
              &&  a(p-2) a(p-1) b(p, p+q) b(p-1, p+q-1) ] \nonumber \\              
              &&  a(p) a(p+1) \dots a(q+1) \nonumber \\
              &&  b(1, q+2) \nonumber \\
              &&  a(2q+1) a(2q) \dots a(p+q+1)
\end{eqnarray}
%%%%%%%%%%%%%%%%%%%%%%%%%%%%%%%%%%%%%%%%%%%%%%%   1 vi)   %%%%%%%%%%%%%%%%%%%%%%%%%%%%%%%%%%%%%%%%%%%%%%%%%%%
vii) $l=1$, $p$ odd, $q+2\leq p\leq 2q+1$, $V(w)=A^pB$, where {\color{\correctedMAu}  \\\phantom{vii)} 
   $w=\underbrace{A}_{[(p-q-1)\wtimes]}B\underbrace{A}_{[(p-q-1)\wtimes]}\underbrace{ABBA}_{[(q-(p-1)/2)\wtimes]}$.}
\begin{eqnarray}\label{1_vii}
\vec{\gamma}_w&\to& a(1) a(2) \dots a(p-q-1) \nonumber \\
              &&  b(p-q, p) \nonumber \\
              &&  a(p-1) a(p-2) \dots a(q+1) \nonumber \\
              &&  [a(q) b(q-1, 2q) b(q, 2q+1) a(q-1) \nonumber \\            
              &&  a(q-2) b(q-3, 2q-2) b(q-2, 2q-1) a(q-3) \nonumber \\
              &&  \vdots \nonumber \\
              &&  a(p-q+1) b(p-q, p+1) b(p-q+1, p+2) a(p-q)]                            
\end{eqnarray}
%%%%%%%%%%%%%%%%%%%%%%%%%%%%%%%%%%%%%%%%%%%%%%%   1 vi)   %%%%%%%%%%%%%%%%%%%%%%%%%%%%%%%%%%%%%%%%%%%%%%%%%%%
viii) $l=1$, $p$ odd, $1\leq p\leq q+1$, $V(w)=A^pB$, where {\color{\correctedMAu}  \\\phantom{viii)} 
   $w=B\underbrace{A}_{[(q-p+1)\wtimes]}\underbrace{ABBA}_{[((p-1)/2)\wtimes]}\underbrace{A}_{[(q-p+1)\wtimes]}$.}
\begin{eqnarray}\label{1_viii}
\vec{\gamma}_w&\to& b(1,q+1) \nonumber \\
              &&  a(q) a(q-1) \dots a(p) \nonumber \\
              &&  [ a(p-1) b(p-2, p+q-1) b(p-1, p+q) a(p-2) \nonumber \\            
              &&  a(p-3) b(p-4, p+q-3) b(p-3, p+q-2) a(p-4) \nonumber \\
              &&  \vdots \nonumber \\
              &&  a(2) b(1,q+2) b(2,q+3) a(1) ] \nonumber \\              
              &&  a(2q+1) a(2q) \dots a(p+q+1)
\end{eqnarray}
%%%%%%%%%%%%%%%%%%%%%%
In the cases iv) - viii) we can obtain the matrices $M$ and $M'$ analogously as in the cases i), ii) and iii). Note that in each case {\color{\correctedTV} except for the first one,} the matrix $M'$ has exactly one non-zero entry in every row and in every column, which implies that $det(M) = \pm 1$.
\end{proof}

\begin{theorem}\label{th_1}
For {\color{\correctedTV} every} odd $k\geq3$ and for {\color{\correctedTV} every} even $d\geq2$ there exists a Cayley digraph of order $2k(\frac{d}{2})^k$, degree $d$ and diameter at most $k$.  
\end{theorem}

\begin{proof}
Let $H$ be an abelian group of order $n\geq1$ and let $\Gamma_k=H^k\rtimes_\varphi D_k$. Let $G=Cay(\Gamma_k,X)$ be the Cayley digraph for the underlying group $\Gamma_k$ and for the generating set $X=\{a(x),b(x)\vert x\in H\}$
{\color{\correctedTV} introduced in Preliminaries.}
Since $\vert X\vert=2n$, the Cayley digraph $G$ has degree $d=2n$ and the order of $G$ is $2kn^k=2k(\frac{d}{2})^k$. To show that $G$ has diameter at most $k$ it is sufficient to show that every element of $\Gamma_k$ is a product of at most $k$ elements of $X$. Let $C\in D_k$. By Lemma \ref{lem_eps_1} for every $C\in D_k$ there is a word $w$ in $D_k$ such that the fiber matrix over $w$ has determinant equal to $\pm1$. By Lemma 
\ref{lem_regular_covers} the fiber over $w$ covers the element $C=V(w)$ of $D_k$, that is every element of $\Gamma_k$ with suffix $C$ is a product of $k$ elements of $X$. Therefore the diameter of $G$ is at most $k$.
\end{proof}

\begin{corollary}\label{cor_1}
For any odd $k\geq3$ and for any $d\geq2$ there exists a Cayley digraph of order $2k\left(\floor{\frac{d}{2}}\right)^k$, degree $d$ and diameter at most $k$.
\end{corollary}

\begin{proof}
Adding one additional generator to the generating set $X$ presented in the proof of Theorem \ref{th_1}, cannot increase the diameter. Therefore we obtain a Cayley digraph of degree $d=2n+1$, order $2kn^k=2k(\frac{d-1}{2})^k$ and diameter at most $k$. 
\end{proof}

\begin{theorem}[Main theorem for Cayley digraphs]\label{th_main_1}
Let $k\geq3$ be odd and let $d\geq\frac{3^k}{2k}+1$. Then $C_{d,k}\geq 2k\lrexp{\floor{\dhalf}}{k}$.
\end{theorem}

\begin{proof}
We show that for $d\geq\frac{3^{k}}{2k}+1$, the Cayley digraphs described in the proofs of Theorem \ref{th_1} and Corollary \ref{cor_1} have 
diameter $k$. The maximum order of a digraph of degree $d$ and diameter $\tilde{k}\leq k-1$ is $M_{d,k-1}=1+d+d^2\dots+d^{k-1}=\frac{d^k-1}{d-1}$, therefore it is sufficient to show that for $d\geq\frac{3^k}{2k}+1$ the orders of the Cayley digraphs are greater than the Moore bound for digraphs of diameter $k-1$. That is, if $2k\lrexp{\floor{\dhalf}}{k}>\frac{d^k-1}{d-1}$, then the Cayley digraphs have diameter exactly $k$.
We distinguish two cases.\\
i) $d\geq2$ is even:\\
From the inequality $2k\lrexp{\dhalf}{k}>\frac{d^k-1}{d-1}$ we get $d>\frac{2^k}{2k}-\frac{1}{2k}\lrexp{\frac{2}{d}}{k}+1$.
Since $-\frac{1}{2k}\lrexp{\frac{2}{d}}{k}<0$, we obtain that if $d\geq \frac{2^k}{2k}+1$, then the Cayley digraph has diameter exactly $k$.\\
ii) $d\geq3$ is odd:\\ 
{\color{\correctedTV} From $2k\lrexp{\dmhalf}{k}>\frac{d^k-1}{d-1}$ we have $d>\lrexp{\frac{d}{d-1}}{k}\frac{2^k}{2k}-\frac{1}{2k}\lrexp{\frac{2}{d-1}}{k}+1$.
Clearly $-\frac{1}{2k}\lrexp{\frac{2}{d-1}}{k}<0$ and $\frac{d}{d-1}\leq\frac32$, thus if $d\geq \frac{3^k}{2k}+1$, the diameter of the Cayley digraph is exactly $k$.}
\end{proof}
%%%%%%%%%%%%%%%%%%%%%%%%%%%%%%%%%%%%%%%%%%%%%%%%%%%%%%%%%%%%%%%%%%%%%%%%%%%%%%%%%%%%%%%%%%%%%%%%%%%%%%%%%%%%
%%%%%%%%%%%%%%%%%%%%%%%%%%%%%%%%%%%%%%%%%%%%%%%%%%%%%%%%%%%%%%%%%%%%%%%%%%%%%%%%%%%%%%%%%%%%%%%%%%%%%%%%%%%%
\subsection{Bipartite Cayley digraphs}\label{subsec_bip_digraphs}
In the next lemma, for any element $A^pB$ ($A^{p+1}$) of the group $D_k$, $k$ even,\\ $p\in\{1,3,5,\dots,k-1\}$, we present a sequence of $k$ generators (a word in $D_k$) whose product is the element with suffix $A^pB$ ($A^{p+1}$) and with the corresponding matrix $M$ with determinant $\pm1$.

\begin{lemma}\label{lem_eps_2}
Let $q\geq2$ be an integer and let $k=2q$. Then for every $C\in\{A^pB,A^{p+1}\vert 1\leq p\leq k-1,\ p\ \mathit{odd}\}$, $C\in D_k$, there is a word $w$ in $D_k$ such that $V(w)=C$ and the fiber matrix $M(w)$ over the word $w$ is a matrix with $det(M)=\pm1$.
\end{lemma}

\begin{proof}
We distinguish four cases: i), ii) for $C=A^{p+1}$ and iii), iv) for $C=A^pB$. In the first two cases we will use the notation $p'=p+1$, $p'$ even.\\\\
%%%%%%%%%%%%%%%%%%%%%%%%%%%%%%%%%%%%%%%%%%%%%%%   2 i)   %%%%%%%%%%%%%%%%%%%%%%%%%%%%%%%%%%%%%%%%%%%%%%%%%%%
i) $q\leq p'\leq 2q$, $V(w)=A^{p'}$, where {\color{\correctedMAu}  \\\phantom{i)} 
   $w=\underbrace{A}_{[(p'-q)\wtimes]}\underbrace{ABBA}_{[(2q-p'/2)\wtimes]}\underbrace{A}_{[(p'-q)\wtimes]}$.}
\begin{eqnarray}\label{2_i}
\vec{\gamma}_w&\to& a(1) a(2) \dots a(p'-q)  \nonumber \\
              && [ a(p'-q+1) b(p'-q+2,p+2) b(p'-q+1,p+1) a(p'-q+2) \nonumber \\            
              &&  a(p'-q+3) b(p'-q+4,p+4) b(p'-q+3,p+3) a(p'-q+4) \nonumber \\
              &&  \vdots \nonumber \\
              &&  a(q-1) b(q,2q) b(q-1, 2q-1) a(q) ] \nonumber \\
              &&  a(q+1) a(q+2) \dots a(p')
\end{eqnarray}
%%%%%%%%%%%%%%%%%%%%%%
{\color{\correctedTV} It can be checked that, in (\ref{2_i}), non-zero entries}
of generators $a(i)$ appear in the rows $1, 2, \dots , p'$.
We remove from $M$ rows and columns in which these non-zero entries appear
to obtain a new square matrix $M'$ of degree $2q-p'$, where $det(M') = \pm det(M)$. Columns of $M'$ correspond to generators $b(j, j')$.
Note that the $j$-th row of $M$ does not appear in $M'$, and
the $j'$-th row of $M$ appears in $M'$. Hence
$M'$ has exactly one non-zero element in each row and in each column, which implies that  $det(M') = \pm1$.\\

{\color{\correctedTV} For the remaining cases we list the corresponding instances and then append a note of how they can be handled.}
\clearpage
ii) $2\leq p'\leq q-1$, $V(w)=A^{p'}$, where {\color{\correctedMAu}  \\\phantom{ii)} 
   $w=B\underbrace{A}_{[(q-p'+2)\wtimes]}\underbrace{BBAA}_{[(p'/2-1)\wtimes]}\underbrace{A}_{[(q-p')\wtimes]}B$.}
\begin{eqnarray}\label{2_ii}
\vec{\gamma}_w&\to& b(1,q+1) \nonumber\\
              &&  a(q) a(q-1) \dots a(p'-1)  \nonumber \\
              && [ b(p'-2, p'+q-2) b(p'-1, p'+q-1) a(p'-2) a(p'-3) \nonumber \\            
              &&  b(p'-4, p'+q-4) b(p'-3, p'+q-3) a(p'-4) a(p'-5) \nonumber \\
              &&  \vdots \nonumber \\
              &&  b(2,q+2) b(3, q+3) a(2) a(1) ] \nonumber \\
              &&  a(2q) a(2q-1) \dots a(p'+q+1) \nonumber \\
              && b(p',p'+q)              
\end{eqnarray}
%%%%%%%%%%%%%%%%%%%%%%%%%%%%%%%%%%%%%%%%%%%%%%%   2 iii)   %%%%%%%%%%%%%%%%%%%%%%%%%%%%%%%%%%%%%%%%%%%%%%%%%%%
iii) $1\leq p\leq q$, $V(w)=A^{p}B$, where {\color{\correctedMAu}  \\\phantom{iii)} 
   $w=B\underbrace{A}_{[(q-p+1)\wtimes]}\underbrace{BBAA}_{[((p-1)/2)\wtimes]}\underbrace{A}_{[(q-p)\wtimes]}$.}
\begin{eqnarray}\label{2_iii}
\vec{\gamma}_w&\to& b(1,q+1) \nonumber\\
              &&  a(q) a(q-1) \dots a(p)  \nonumber \\
              && [ b(p-1, p+q-1) b(p, p+q) a(p-1) a(p-2) \nonumber \\            
              &&  b(p-3, p+q-3) b(p-2, p+q-2) a(p-3) a(p-4) \nonumber \\
              &&  \vdots \nonumber \\
              &&  b(2,q+2) b(3, q+3) a(2) a(1) ] \nonumber \\
              &&  a(2q) a(2q-1) \dots a(p+q+1)
\end{eqnarray}
%%%%%%%%%%%%%%%%%%%%%%%%%%%%%%%%%%%%%%%%%%%%%%%   2 iv)   %%%%%%%%%%%%%%%%%%%%%%%%%%%%%%%%%%%%%%%%%%%%%%%%%%%
iv) $q+1\leq p\leq 2q-1$, $V(w)=A^{p}B$, where {\color{\correctedMAu}  \\\phantom{iv)} 
   $w=\underbrace{A}_{[(p-q)\wtimes]}\underbrace{AABB}_{[(q-(p+1)/2)\wtimes]}\underbrace{A}_{[(p-q+1)\wtimes]}B$.}
\begin{eqnarray}\label{2_iv}
\vec{\gamma}_w&\to& a(1) a(2) \dots a(p-q)  \nonumber \\
              && [ a(p-q+1) a(p-q+2) b(p-q+3,p+3) b(p-q+2,p+2) \nonumber \\            
              &&  a(p-q+3) a(p-q+4) b(p-q+5,p+5) b(p-q+4,p+4) \nonumber \\
              &&  \vdots \nonumber \\
              &&  a(q-2) a(q-1) b(q,2q) b(q-1, 2q-1) ] \nonumber \\
              &&  a(q) a(q+1) \dots a(p) \nonumber \\
              && b(p-q+1,p+1)
\end{eqnarray} 

{\color{\correctedTV} In a way similar to what was presented  in the case i)}
one can show that determinants of matrices {\color{\correctedTV} corresponding} to the products (\ref{2_ii}), (\ref{2_iii}) and (\ref{2_iv}) are 
$1$ or $-1$.
\end{proof}

\begin{theorem}\label{th_2}
For {\color{\correctedTV} every} odd $k\geq5$ and for {\color{\correctedTV} every} even $d\geq2$ there exists a bipartite Cayley digraph of order $2(k-1)\lrexp{\dhalf}{k-1}$, 
degree $d$ and diameter at most $k$.  
\end{theorem}

\begin{proof}
Let $k'=k-1$, let $n\geq1$ be an integer, let $H$ be an {\color{\correctedTV} abelian group} of order $n$ and let $\Gamma_{k'}=H^{k'}\rtimes_\varphi D_{k'}$. Let $G=Cay(\Gamma_{k'},X)$ be the Cayley digraph for the underlying group $\Gamma_{k'}$ and for the generating set $X=\{a(x),b(x)\vert x\in H\}$
{\color{\correctedTV} introduced in Preliminaries.} Since $\vert X\vert=2n$, the Cayley digraph has degree $d=2n$ and the order of the Cayley digraph is $2k'n^{k'}=2k'(\frac{d}{2})^{k'}$. To show that $G$ has diameter at most $k$ it is sufficient to show that every element of $\Gamma_{k'}$ is a product of at most $k$ elements of $X$. 
Let $S=\{A^pB,A^{p+1}\vert 1\leq p\leq k',\ p\ \mathit{odd}\}$ and let $C\in S$.
By Lemma \ref{lem_eps_2} for every $C\in S$ there is a word $w$ in $D_{k'}$ such that the fiber matrix over $w$ has determinant equal to $\pm1$. By Lemma \ref{lem_regular_covers} the fiber over $w$ covers the element $C=V(w)$ of $S$, that is every element of $\Gamma_{k'}$ with suffix $C\in S$ is a product of $k'$ elements of $X$. 
Since generators from $X$ have suffices either $A$ or $B$, no vertex with suffix $A^l$ is adjacent to the vertex with suffix $A^m B$ if one of $l, m$ is even and the other one is odd. Also no vertex with suffix $A^l B^i$ is adjacent to the vertex with suffix $A^m B^i$ for $i=0$ or $1$, if both, $l, m$ are even, or if both $l, m$ are odd. Therefore 
{\color{\correctedTV} the}
Cayley digraph $G$ is bipartite. Note that in a bipartite Cayley digraph, if one can obtain any vertex in one {\color{\correctedTV} part} as a product of $k'$ generators, then all vertices in the other {\color{\correctedTV} part} can be obtained as products of (at most) $k'+1=k$ generators. Therefore the diameter of $G$ is at most $k$.
\end{proof}

\begin{corollary}\label{cor_2}
For {\color{\correctedTV} every} odd $k\geq5$ and for {\color{\correctedTV} every} $d\geq2$ there exists a bipartite Cayley digraph of 
order $2(k-1)\lrexp{\floor{\dhalf}}{k-1}$, degree $d$ and diameter at most $k$.
\end{corollary}

\begin{proof}
We add one additional generator with suffix $A$ to the generating set $X$ given in the proof of Theorem \ref{th_2} to obtain a Cayley digraph which is still bipartite and has the diameter at most $k$. 
{\color{\correctedTV} We thus obtain}
a bipartite Cayley digraph of degree $d=2n+1$, 
order $2(k-1)n^{k-1}=2(k-1)\lrexp{\floor{\dhalf}}{k-1}$ and diameter at most $k$. 
\end{proof}

\begin{theorem}[Main theorem for bipartite Cayley digraphs]\label{th_main_2}
Let $k\geq5$ be odd and let $d\geq\frac{3^{k-1}}{k-1}+1$. Then $BC_{d,k}\geq 2(k-1)\lrexp{\floor{\dhalf}}{k-1}$.
\end{theorem}

\begin{proof}
Let $d\ge \frac{3^{k-1}}{k-1}+1$. 
Let us prove that the bipartite Cayley digraphs described in the proofs of Theorem \ref{th_2} and Corollary \ref{cor_2} are of diameter $k$. 
Since $k$ is odd,
the order of a bipartite digraph of diameter $\tilde{k}\leq k-1$ cannot exceed the bound 
$\mb{k-1}=2d(1+d^2+d^4+\dots+d^{k-2})=2d\frac{d^{k-1}-1}{d^2-1}$. Thus it suffices to show that if 
$d\geq\frac{3^{k-1}}{k-1}+1$, then the orders of our bipartite Cayley digraphs are greater than the Moore bound for bipartite digraphs of diameter $k-1$. That is, if $2(k-1)\lrexp{\floor{\dhalf}}{k-1}>2d\frac{d^{k-1}-1}{d^2-1}$, then the Cayley digraphs are of diameter exactly $k$.
Let us consider two cases.\\
i) $d\geq2$ is even:\\
From $2(k-1)\lrexp{\dhalf}{k-1}>2d\frac{d^{k-1}-1}{d^2-1}$ we have $d>\frac{2^{k-1}}{k-1}-\frac{1}{k-1}\lrexp{\frac{2}{d}}{k-1}+\frac{1}{d}$.
Since $-\frac{1}{k-1}\lrexp{\frac{2}{d}}{k-1}<0$ and $\frac{1}{d}<1$, we get that if $d\geq \frac{2^{k-1}}{k-1}+1$, then the diameter of the Cayley digraph is exactly $k$.\\
ii) $d\geq3$ is odd:\\
From the inequality $2(k-1)\lrexp{\dmhalf}{k-1}>2d\frac{d^{k-1}-1}{d^2-1}$ we get $d>\lrexp{\frac{d}{d-1}}{k-1}\frac{2^{k-1}}{k-1}-\frac{1}{k-1}\lrexp{\frac{2}{d-1}}{k-1}+\frac{1}{d}$.
Clearly $-\frac{1}{k-1}\lrexp{\frac{2}{d-1}}{k-1}<0$, $\frac{1}{d}<1$ and $\frac{d}{d-1}\leq\frac32$, thus if $d\geq \frac{3^{k-1}}{k-1}+1$, then the Cayley digraph is of diameter exactly $k$.
\end{proof}
%%%%%%%%%%%%%%%%%%%%%%%%%%%%%%%%%%%%%%%%%%%%%%%%%%%%%%%%%%%%%%%%%%%%%%%%%%%%%%%%%%%%%%%%%%%%%%%%%%%%%%%%%%%%
%%%%%%%%%%%%%%%%%%%%%%%%%%%%%%%%%%%%%%%%%%%%%%%%%%%%%%%%%%%%%%%%%%%%%%%%%%%%%%%%%%%%%%%%%%%%%%%%%%%%%%%%%%%%
\section{{\color{\correctedTV} Conclusion}}\label{conclusion}

In this paper we construct the largest known Cayley digraphs of odd diameter and large degree using semidirect products
$\Gamma_k=H^k\rtimes_\varphi D_k$.
Simpler semidirect products $H^k\rtimes_\varphi Z_k$ have been used before.
It would be interesting to consider $H^k\rtimes_\varphi P_k$ for various subgroups $P_k$ of $Sym(k)$. The most difficult part in this research is often proving that the diameter of a Cayley digraph is exactly/at most $k$.

Note that the product $H^k\rtimes_\varphi Sym(k)$ cannot be used to construct Cayley digraphs $C(H^k\rtimes_\varphi Sym(k), X)$ of degree $d$ and diameter $k$ unless the genarating set $X$ is quite large (which would decrease the order of a Cayley digraph in terms of $d$ and $k$).
For example if $|X| = cn$ for some $c > 0$, then the order of $C(H^k\rtimes_\varphi Sym(k), X)$
is $k!n^k = k! \frac{d^k}{c^k}$. Since the order of a Cayley digraph of degree $d$ and diameter $k$ cannot exceed the Moore bound $1 + d + d^2 + \dots + d^k$, it follows that for $d \to \infty$ one must have $\frac{k!}{c^k} \le 1$.

Let us also mention that all known construction of Cayley digraphs are very far from the theoretical upper 
{\color{\correctedTV} bound, leaving space}
for future research.

\section*{Acknowledgements}
The authors would like to express their gratitude to the referees for all the valuable and constructive comments.

The research of the first author was supported by VEGA Research Grant No. 1/0811/14 and by the Operational
Programme ‘Research \& Development’ funded by the European Regional Development Fund through
implementation of the project ITMS 26220220179.
Research of the second author was supported by the National Research Foundation of South Africa, Grant numbers: 91499, 90793.
%%%%%%%%%%%%%%%%%%%%%%%%%%%%%%%%%%%%%%%%%%%%%%%%%%%%%%%%%%%%%%%%%%%%%%%%%%%%%%%%%%%%%%%%%%%%%%%%%%%%%%%%%%%%
%%%%%%%%%%%%%%%%%%%%%%%%%%%%%%%%%%%%%%%%%%%%%%%%%%%%%%%%%%%%%%%%%%%%%%%%%%%%%%%%%%%%%%%%%%%%%%%%%%%%%%%%%%%%

%%%%%%%%%%%%%%%%%%%%%%%%%%%%%%%%%%%%%%%%%%%%%%%%%%%%%%%%%%%%%%%%%%%%%%%%%%%%%%%%%%%%
\end{document}